\documentclass[10pt,a4paper,reqno]{amsart}
\usepackage{a4wide,amsfonts,amsmath,amssymb,amsthm,enumerate,color,hyperref,amsthm,epsfig,graphicx}
\usepackage[latin1]{inputenc}
\usepackage{xcolor}
\newcommand{\chg}{}
\newcommand{\chge}{}

\newcommand{\argmin}{{\rm Argmin}\;}

\newcommand{\PP}{\mathcal P}

\renewcommand{\P}{\mathcal P}
\newcommand{\J}{\mathcal J}

\newcommand{\C}{\mathcal C}
\newcommand{\M}{{\mathcal W}}
\renewcommand{\S}{{\mathcal S}}
\newcommand{\R}{\mathbb R}
\newcommand{\RR}{\mathbb R}
\newcommand{\dd}{\;{\rm d}}

\newcommand{\supp}{{\rm Supp}\,}
\newcommand{\dom}{{\rm dom}\,}
\newcommand{\dist}{{\rm dist}\,}
\newcommand{\proj}{{\rm proj}\,}

\newcommand{\rhor}{\rho^*}
\newcommand{\cg}{c_g}
\newcommand{\cf}{c_f}
\newcommand{\Pd}{\mathcal{P}_2(\R^d)}
\newcommand{\go}{\Gamma_{\text{\rm opt}}}

\newtheorem{theorem}{Theorem}
\newtheorem{proposition}{Proposition}
\newtheorem{lemma}{Lemma}

\newtheorem{claim}{Claim}
\theoremstyle{remark}
\newtheorem{definition}{Definition}
\newtheorem{remark}{Remark}

\newtheorem{property}{Property}

\title[A family of functional inequalities]{A family of functional inequalities: \L ojasiewicz inequalities and displacement convex functions}
\author{Adrien Blanchet \& J\'er\^ome Bolte}
\address{Toulouse School of Economics, Universit\'e Toulouse Capitole, Manufacture des Tabacs, 21 all\'{e}e de Brienne, 31015, Toulouse, France. 
}
\email{adrien.blanchet@tse-fr.eu, jerome.bolte@tse-fr.eu}

\date{\today}

\begin{document}

\maketitle

\begin{abstract}
For displacement convex functionals in the probability space equip\-ped with the Monge-Kantorovich metric we prove the equivalence between the gradient and functional type \L oja\-sie\-wicz inequalities. \chg{We also discuss the more general case of $\lambda$-convex functions and we provide a general convergence theorem for the corresponding gradient dynamics. Specialising our results to the Boltzmann entropy, we recover  Otto-Villani's theorem asserting the equivalence between logarithmic Sobolev and Talagrand's inequalities. The choice of power-type entropies shows a new  equivalence between Gagliardo-Nirenberg inequality and a nonlinear Talagrand inequality. Some nonconvex results and other types of equivalences are discussed. }\end{abstract}


\section{Introduction}
\chg{\L ojasiewicz inequalities are known to be extremely powerful tools for studying the long-time behaviour of dissipative
systems in an Euclidean or 
Hilbert space, see e.g., \cite{simon83,haraux,chill03,bolte10} and references therein. Their connection with the asymptotics  of gradient flows comes from the fact that one of this inequality asserts that the underlying energy can be rescaled near critical points into a sharp function\footnote{See Remark~\ref{r:loja} for further explanations}. A consequence of this inequality is that gradient curves can be shown to have finite length through the choice of an adequate Lyapunov function, see~\cite{loja65,kurd,bolte10}\footnote{See also  the proof of Theorem~\ref{thm:main}}.

In parallel the study of large time asymptotics of various PDEs, also based on energy techniques, was developed in close conjunction with functional inequalities. A classical study protocol is to evidence  Lyapounov functionals and use functional inequalities to derive quantitative contractive properties of the flow. The heat equation provides an elementary but illustrative example of this approach: the Boltzmann entropy gives a Lyapunov functional while  the Logarithmic-Sobolev inequality  ensures the exponential convergence of the solution curve to a self-similar profile. Numerous applications of these techniques, as well as their
stochastic counterparts, can be found in  e.g., \cite{bakry85,carrillo06,ane,cordero,bobkov01,djellout04}. Standard references for functional inequalities are for instance \cite{carrillo03,ledoux05,ledoux93,gozlan2010}.


%
%
%
%

In this article we show  that the joint use of  metric gradient flows and \L ojasiewicz inequalities 
allows for a systematic and transparent treatment of these evolution equations.   In this regard the ``Riemannian structure"  of  the set of probability measures endowed with the Monge-Kantorovich distance (see \cite{otto01,ags,gangbo96,villani03,jordan}) plays a fundamental role in our approach. It allows in particular to interpret some PDEs as gradient flows, like Fokker-Planck, porous medium, or fast diffusion equations, and it provides a setting sufficiently rich to formulate precisely  \L ojasiewicz inequalities.

In a first step we indeed introduce two types of \L ojasiewicz inequalities in the probability space equipped with the Monge-Kantorovich distance. One of these inequalities  is a growth measure of the energy functional with respect to the Monge-Kantorovich distance to stationary points, while the other  provides a relationship between  the values of the energy and its slope.  The latter is called the {\em gradient \L ojasiewicz inequality}. In this functional setting, both inequalities can be viewed  as families of abstract functional inequalities. 	
We prove their equivalence in the case of convex functionals.   Specialising our  results to Boltzmann's entropy we recover Otto-Villani's theorem~\cite{ottovil} stating the  equivalence between the logarithmic 
Sobolev and Talagrand inequalities. We also prove a new equivalence between Gagliardo-Sobolev inequality and a nonlinear Talagrand type inequality~\eqref{CKOI}.  For general $\lambda$-convex functionals the gradient \L ojasiewicz gradient inequality merely implies the growth around the minimisers set, the reverse implication is in general false. This difficulty can be \chge{felt} through the fact that Talagrand's inequality is not known to imply logarithmic Sobolev inequality for \chge{a general non-convex\footnote{\chge{However it holds for potential whose Hessian is bounded from below, see \cite[Corollary 3.1]{ottovil}}} confinement potential}, see \cite[Section~9.3.1, p.292]{villani03}. \chge{In Section~\ref{s:ineq} we discuss further this issue and provide a new proof that the reverse implication holds for convex potentials with $L^\infty$ variations},~\cite[Corollary~2.2]{ottovil}.

In Theorem \ref{t:conv}, we use \L ojasiewicz inequalities in a classical way to analyse the convergence and the rate of evolution equations. Our results do not subsume the uniqueness of a minimiser/critical point, neither the convexity of the potential. We recover many classical results in a unified way and we provide new insights into the stabilisation phenomena  in the presence of a continuum of equilibria.

\bigskip

The structure of the paper is as follows. In Section~\ref{s:loja}, we state  two fundamental types of \L ojasiewicz inequalities and study their equivalence. General convergence 
rate results for subgradient systems  are also provided.  In Section~\ref{s:ineq} we translate these inequalities into functional 
inequalities in the case of the Boltzmann and non-linear entropies. 

\smallskip

\noindent
Notations, definitions and classical results on Monge-Kantorovich metric and optimal transportation are postponed to  Section~\ref{s:app}.
}
\section{\L ojasiewicz inequalities for displacement convex functions}\label{s:loja}

\subsection{Main results}\chg{
\noindent Let $\PP_2(\R^d)$ be the  set of probability measures in $\R^d$ with bounded se\-cond moments.   In the sequel $\J:\PP_2(\R^d) \to (-\infty,+\infty]$  is a lower semi-continuous function which is $\lambda$-convex along generalised geodesics with $\lambda$ in $\R$. 
\bigskip

}
\noindent
{\bf Two  inequalities of \L ojasiewicz type.} 
Assume that $\J$ has at least a minimiser. Set $\hat \J=\min\{\J[\rho]:\rho\in\P_2(\R^d)\}$. Fix $r_0\in (\hat \J,+\infty]$ and $\theta \in (0,1]$.  We consider the two following properties:

\begin{property}[\L ojasiewicz gradient property]
 There exists   $\cg>0$  such that for all $\rho \in \PP_2(\R^d)$, 
  \begin{equation}\label{eq:lg}
   \hat \J< \J[\rho]< r_0 \quad  \Rightarrow  \quad  \forall \nu \in \partial \J[\rho]\;,\; \cg\left(\J[\rho]-\hat \J\right)^{1-\theta} \le \|\nu\|_{\rho},
  \end{equation}
\chg{where $\partial \J$ stands for the subdifferential of $\J$ in the probability space equipped with the Monge-Kantorovich metric and $\|\cdot\|_{\rho}$ is the norm in $L^2_\rho(\R^d)$, see Section~\ref{s:app}.} 
\end{property}
\begin{property}[Functional \L ojasiewicz  inequality]
  There exists $\cf>0$  such that for all $\rho \in \PP_2(\R^d)$,  
  \begin{equation}\label{eq:lf}
 \hat \J< \J[\rho]< r_0 \quad \Rightarrow \quad  \cf \, \M_2(\rho,\argmin \J)^{{1}/{\theta}} \leq \,\J[\rho]- \hat \J.
  \end{equation}
\end{property}

\medskip 
\begin{remark}\label{r:loja}\chg{
(a) A formal geometric interpretation of \eqref{eq:lg} is as follows: assume $\J$ sufficiently smooth to rewrite the inequality \eqref{eq:lg} as
$$\|\nabla \left(\J-\hat \J\right)^\theta[\rho]\|_\rho\geq \theta\cg,$$ 
for any $\rho$ such that  $\hat \J< \J[\rho]< r_0$.  \chge{With this reformulation, we see} that the norm of the gradient of $(\J-\hat \J)^{\theta}$ is  bounded away from $0$ for non-critical measures. The obtained reparameterization $(\J-\hat \J)^{\theta}$ of $\J$ is called {\em sharp} \chge{in reference to the sharpness of its profile near its critical set}. It provides   a Lyapunov function for the corresponding gradient system with strong decrease rate properties \chge{reflecting the $V$-shape of the rescaled energy} (see the proof of  Theorem~\ref{thm:main}).\\
 The second inequality \eqref{eq:lf} is a classical growth inequality around stationary points.}\\
(b) Original inequalities for analytic and subanalytic functions can be found in the IHES lectures~\cite{loja65} by \L ojasiewicz. Several generalisations of gradient inequalities followed, see in particular~\cite{simon83,kurd,BDLS}.\\
(c) In this infinite dimensional setting, the above \L ojasiewicz inequalities should be thought as families of functional inequalities. Formal connections with the Talagrand, logarithmic  Sobolev and Gagliardo-Nirenberg inequalities are provided in Section~\ref{s:ineq}.\\
\chg{(d) When $r_0=+\infty$ the above inequalities are {\em global inequalities} which will be the case of most examples provided in Section~\ref{s:ineq}.\\
(e) Note that when $\J$ is not convex, Property~\eqref{eq:lg} implies that there are no critical values between $\hat \J$ and $r_0$, while~\eqref{eq:lf} does not preclude this possibility.}
\end{remark}
\noindent



\medskip

We now state our two main theorems.
\chg{\begin{theorem}[Equivalence between \L ojasiewicz inequalities in $\PP_2(\R^d)$]\label{thm:main}
  Let $\J$ be a proper lower semi-continuous functional which has at least a global minimiser.  
  \begin{itemize} 
  \item[(i)] If $\J$ is $\lambda$-convex along generalised geodesics for some $\lambda$ in $\R$, then the \L ojasiewicz gradient property~\eqref{eq:lg} implies the \L ojasiewicz functional property~\eqref{eq:lf}.
   \item[(ii)] If $\J$ is  convex along generalised geodesics, then the \L ojasiewicz gradient property~\eqref{eq:lg} and the \L ojasiewicz functional property~\eqref{eq:lf} are e\-qui\-va\-lent.
   \end{itemize}
\end{theorem}}

\begin{remark}\label{r:constant} (a) \chg{Note that the quantities $r_0,\theta$ are conserved in both cases. } Theorem \ref{thm:main}  provides a relationship between $\cf$ and $\cg$ but the optimality of the constant might be lost\footnote{See Section~\ref{sec:nrj} where optimality of the constant is preserved by $\eqref{eq:lg}\Rightarrow \eqref{eq:lf}$}. Indeed our proof gives 
$\cg=\cf^{\theta}$ when one establishes $\eqref{eq:lf}\Rightarrow \eqref{eq:lg}$, while $\cf=(\theta\cg)^{{1}/{\theta}}$ when $\eqref{eq:lg}\Rightarrow \eqref{eq:lf}$ is proved. \\
Yet, when $\theta=1$ and $\J$ is convex the equivalence between~\eqref{eq:lg} and~\eqref{eq:lf} holds with  
$\cf=\cg$.\\
(b) Assume for simplicity that $\hat \J=0$. The extension of the gradient \L ojasiewicz inequality by Kurdyka \cite{kurd} would write in our setting:
 \begin{equation}\label{kurd}
    0<\J[\rho]< r_0 \quad  \Rightarrow  \quad  \forall \nu \in \partial \J[\rho]\;,\; 1 \le \varphi'(\J[\rho])\|\nu\|_{\rho},\;
  \end{equation}
  for some $\varphi \in C^0[0,r_0)\cap C^1(0,r_0)$ \chg{such that} $\varphi(0)=0$ and $\varphi'>0$ on $(0,r_0)$. \chg{Note that \eqref{eq:lg} is nothing but inequality \eqref{kurd}} with $\varphi(s)=s^\theta/({\theta \cg})$.  Mimicking the proof of Theorem~\ref{thm:main} one can establish that~\ref{kurd}  implies:
\begin{equation}\label{err}
0< \J[\rho]< r_0 \quad \Rightarrow \quad  \varphi^{-1}(\M_2(\rho,\argmin \J))\leq \,\J[\rho].
  \end{equation}
Yet  \eqref{kurd} and \eqref{err} are not in general equivalent. One can build a $C^2$ convex coercive function in $\R^2$ which satisfies \eqref{err} but not \eqref{kurd}, \chg{whatever the choice of $\varphi$, see \cite{bolte15}}. The construction is fairly complex, involves highly oscillatory behaviour of level sets and originates in \cite{bolte10}.  This limitation helps to understand the discrepancy between constants mentioned in item (a) of this remark.\\
\chg{(c) An Hilbertian version of the above result was provided in \cite{bolte15} in view of studying the complexity of first-order methods.}
\end{remark}

Our second result concerns the flow of $-\partial \J$ and is simply a Monge-Kantorovich version of the classical Hilbertian  results. Recall that $\J$ is a  lower semi-continuous functional $\J:\PP_2(\R^d)\to (-\infty,+\infty]$ which is \chg{ $\lambda$-convex along generalised geodesics}. For  $\rho_0\in \dom \J$, we consider absolutely continuous solutions $\rho:[0,+\infty)\to \PP_2(\R^d)$ to the subgradient dynamics:
\begin{equation}\label{subgrad}
\frac{\!\dd }{\!\dd t} \rho(t)+\partial \J[\rho(t)]\ni0, \text{ for almost every $t$ in }(0,+\infty),
\end{equation}
with initial condition $\rho(0)=\rho_0$. Such a solution exists and is unique see 
\cite[Theorem 11.3.2, p.305]{ags}  and \cite[Theorem 11.1.4, p.285]{ags}. Moreover by the energy identity \cite[Theorem 11.3.2, p.305]{ags}, one has:
\begin{align}\label{dec}
 & \J[\rho(\cdot)] \text{ is non-increasing, absolutely continuous,}  \\
 \label{prop}
 & \chg{\frac{\!\dd}{\!\dd t} \J[\rho(t)]=-\left |\frac{\!\dd \rho }{\!\dd t} \right |^2 \!(t) \text{ a.e. on }(0,+\infty).}   
\end{align}

\medskip

As recalled in the introduction it is well known that functional inequalities are key tools for the asymptotic study of dissipative systems of gradient type. If  \L ojasiewicz inequalities are seen as  families of functional inequalities, the theorem below could  be considered as an abstract principle to deduce the convergence of a gradient flow from a given functional inequality\footnote{See  Section~\ref{s:ineq} in which some classical functional inequalities are interpreted as \L ojasiewicz inequalities}.
  
\begin{theorem}[Global convergence rates and functional inequalities]\label{t:conv} \chg{Consider $\J$ a proper lower semi-continuous  $\lambda$-convex functional. Assume $\J$  has at least a minimiser and satisfies the gradient inequality \eqref{eq:lg}.  Let $\rho_0\in [\hat \J\leq \J<r_0]$, and consider a trajectory 
of~\eqref{subgrad} starting at $\rho_0$.
}
Then this trajectory has a finite length and 
converges for the Monge-Kantorovich metric to a minimiser $\rho_{\infty}$ of $\J$.

Moreover the following estimations hold:
\begin{enumerate}
\item[(i)]  If  $\theta\in (0,1/2),$ then for all $t\geq0$,
\begin{align*}
\J[\rho(t)]-\hat \J\leq \left\{ (\J[\rho_0]-\hat \J)^{2\theta-1}+\cg^2(1-2\theta)t  \right\}^{-\frac{1}{1-2\theta}},\\
\M_2(\rho(t),\rho_\infty)\leq \frac{1}{\cg \theta} \left\{  (\J[\rho_0]-\hat \J)^{2\theta-1}+\cg^2(1-2\theta)t  \right\}^{-\frac{\theta}{1-2\theta}}.
\end{align*}

\item[(ii)] If  $\theta={1}/{2},$ then for all $t\geq0$,
\begin{align*}
\J[\rho(t)]-\hat \J\leq \chge{\left(\J[\rho_0]-\hat\J\right)}\exp[-\cg^2t],\\\
\M_2(\rho(t),\rho_\infty)\leq \frac{2}{\cg}\chge{ \sqrt{\J[\rho_0]-\hat\J}}\,\exp{\Big[-\frac{\cg^2t}{2}\Big]}.
\end{align*}
\item[(iii)]  If  $\theta\in ({1}/{2},1]$ we observe a finite time stabilisation: The final time is smaller than
$$T=\frac{(\J[\rho_0]-\hat \J)^{2\theta-1}}{\cg^2(2\theta-1)}.$$
When $t$ is in $[0,T]$
\begin{align*}
\J[\rho(t)]-\hat \J\leq \left\{(\J[\rho_0]-\chg{\hat \J})^{2\theta-1}-\cg^2(2\theta-1)t  \right\}^{\frac{1}{2\theta-1}},\\\
\M_2(\rho(t),\rho_\infty)\leq \frac{1}{\cg\theta} \left\{ (\J[\rho_0]-\chg{\hat \J})^{2\theta-1}-\cg^2(2\theta-1)t  \right\}^{\frac{\theta}{2\theta-1}}.
\end{align*}
\end{enumerate}

\end{theorem}

\begin{remark}\chg{
(a) A  fundamental feature of this convergence result is that {\em it does not subsume the uniqueness of a minimiser}  which is uncommon in the domain. \\
(b) It can be proved that  the generalised \L ojasiewicz gradient inequality of Remark~\ref{r:constant}~(b), often called Kurdyka-\L ojasiewicz property, allows as well to study the convergence of gradient system~\eqref{subgrad}. \\
(c) By Theorem \ref{thm:main}, when $\J$ is convex it is sufficient to assume \eqref{eq:lf} instead of \eqref{eq:lg}. This fact is quite important in practice since  \eqref{eq:lf} is   a ``zero-order" functional inequality\footnote{It only involves the value of the function $\J$ and no higher order information} and is in general easier to derive than \eqref{eq:lg}.\\
} \end{remark}



\subsection{Proofs of the main results }
\begin{lemma}[Slope and convexity]\label{convineq}

Let  $\J$ be a proper lower semi-continuous convex functional and \chg{$\mu \in \dom \J$}, $\chg{\nu} \in \dom \partial \J$,   two distinct probabilities in $\R^d$. Then
\begin{equation}
\frac{\J[\nu]-\J[\mu]}{\M_2(\mu,\nu)}\leq \|\partial^0\J[\nu]\|_{\nu}.
\end{equation}
\end{lemma}
\begin{proof} Set $\alpha=\M_2(\mu,\nu)$. Let $[0,\alpha]\ni t\to \rho_t:=\mu_{t/\alpha}$ where $\mu_t$ is \chg{a constant speed geodesic between $\mu$ and $\nu$ (indeed $\P_2(\RR^d)$ is a geodesic space \cite[Theorem 2.10, p.35]{users})}.  Since $t\mapsto \J[\rho_t]$ is convex, we have 
$$\frac{\J[\rho_\alpha]-\J[\rho_0]}{\alpha}\leq \limsup_{\tau\to \alpha}\frac{\J[\rho_{\alpha}]-\J[\rho_{\tau}]}{\alpha-\tau}.$$
\chg{As $t\mapsto\mu_{t}$ is a constant speed geodesic,} $\M_2(\rho_t,\rho_{\tau})=|{t}/{\alpha}-{\tau}/{\alpha}|\M_2(\mu,\nu)=|t-\tau|$ for all time $(t,\tau)\in[0,\alpha]^2$, the above can be rewritten as \chg{
$$\frac{\J[\nu]-\J[\mu]}{\M_2(\mu,\nu)}\leq \limsup_{\tau\to \alpha}\frac{\J[\nu]-\J[\rho_{\tau}]}{\M_2(\nu,\rho_{\tau})}\leq |\nabla| \J[\nu]\leq\|\partial^0\J[\nu]\|_{\nu},$$
where the last inequality follows from \eqref{pente}.}\end{proof}

$\:$\\
\noindent
{\bf Proof of Theorem~\ref{thm:main}.} With no loss of generality, we assume that $\hat \J=\min \J=0$.
\paragraph{$\bullet$ Proof of (i)} We thus prove  \eqref{eq:lg} $\Rightarrow$ \eqref{eq:lf} \chg{without using convexity}. 
Take $\rho_0$ with  $\J[\rho_0]\in (0,r_0)$ and consider the dynamics
\begin{equation}\label{eq:dynl}
\frac{\!\dd }{\!\dd t} \rho(t) + \partial \J[\rho(t)] \ni 0\; \text{ with }\rho(0)=\rho_0.
\end{equation}
Set $\bar t=\sup\{t:\J[\rho(\tau)]>0,\forall \tau \in [0,t)\}$. If $\bar t<+\infty$, the continuity of $\J[\rho(\cdot)]$ ensures that $\J[\rho(\bar t)]=0$. By \eqref{dec}, one has $\J[\rho(t)]=0$ for all $t \geq \bar t$. \chge{By integrating \eqref{prop} over $(\bar t,\tau)$, with $\tau\geq \bar t$, one obtains}
\begin{equation}\label{stab}
 \J[\rho(\bar t)]-\J[\rho(\tau)]=\int _{\bar t}^\tau \left |\frac{\!\dd \rho}{\!\dd s} \right |^2(s) \dd s= 0.
\end{equation}
\chge{where the last equality comes from the ``lazy selection" principle \cite[Theorem 11.3.2 (ii)]{ags}. The latter indeed implies  that the velocity coincides for almost all $t\in (0,+\infty)$ with the minimum norm subgradient which is $0$ in this case.}
Whence $\rho(t)=\rho(\bar t)$ for all $t\geq \bar t$. 

We now consider the case when $t< \bar t$ so that $\J[\rho(t)]>0$. By the chain rule, we have
\begin{equation*}
  -\frac{\!\dd}{\!\dd t} \left[\J[\rho(t)]^\theta \right]=-\theta \J[\rho(t)]^{\theta -1} \frac{\!\dd}{\!\dd t}\J[\rho(t)] \text{ a.e. on $(0,\bar t)$.} 
\end{equation*}	
As $\rho$ follows the dynamics~\eqref{eq:dynl}, we have by \eqref{prop}
\begin{equation*}
  -\frac{\!\dd}{\!\dd t} \left[\J[\rho(t)]^\theta \right]=\theta \J[\rho(t)]^{\theta -1}\left | \frac{\!\dd \rho }{\!\dd t} \right |^2(t)\;\text{a.e. on $(0,\bar t)$}.
  \end{equation*}
Using \L ojasiewicz gradient property \eqref{eq:lg}, we obtain
\begin{equation*}
  -\frac{\!\dd}{\!\dd t} \left[\J[\rho(t)]^\theta \right]\ge \cg \theta  \left | \frac{\!\dd \rho}{\!\dd t} \right |^2(t)\left\| \nu(t)\right\|_{\rho}^{-1},\;
\end{equation*}
for any $\nu(t)\in \partial \J[\rho(t)]$. \chg{From \cite[Theorem 8.3.1, p.183]{ags}, we have for almost every $t$: 
$$\left | \frac{\!\dd \rho }{\!\dd t} \right |(t)=\left\| \frac{\!\dd }{\!\dd t} \rho(t)\right\|_{\rho}.$$}
Using the gradient dynamics one may take $\nu(t)=\chg{-}\dd\rho / \dd t$ for almost all $t$, to obtain
\begin{equation}\label{estim}
  -\frac{\!\dd}{\!\dd t} \left[\J[\rho(t)]^\theta \right] \ge \cg \theta \left |\frac{\!\dd \rho}{\!\dd t} \right|(t)\;.
\end{equation}
Integrating between $0$ and $t\in [0,\bar t)$ and using the absolute continuity of $\J[\rho(\cdot)]$, we obtain
\begin{equation*}
 0 \le \int_0^t \left |\frac{\!\dd\rho}{\!\dd s}  \right |(s) \dd s\le\frac{1}{\cg \theta} \left( \J^{\theta}[\rho_0]-\J^{\theta}[\rho(t)] \right)\;.
\end{equation*}
As a consequence of \eqref{stab} this yields
\begin{equation}\label{macha}
  \int_0^\infty \left |\frac{\!\dd \rho}{\!\dd s}\right |(\chg{s})\dd s \le\frac{1}{\cg \theta} \J^{\theta}[\rho_0],
\end{equation}
which implies in particular that $\left |{\!\dd } \rho/ {\!\dd s}\right | (s)$ is bounded in $L^ 1(0,\infty)$.

\begin{claim}\label{conv} 
  If an absolutely continuous  curve $\R_+\ni t\to \mu(t)\in \PP_2(\R^d)$ satisfies
  \begin{equation*}
     \int_0^\infty \left |\frac{\!\dd \mu}{\!\dd t} \right | (t)\dd t < \infty
  \end{equation*}
then $\mu$ converges  to some $\bar \mu$ as $t\to\infty$ in the  sense of the Monge-Kantorovich metric.
\end{claim}
\noindent
{\em Proof of Claim~\ref{conv}.}
\chg{Simply observe that the absolute continuity and the definition of the metric derivative implies that
\begin{eqnarray}
\M_2(\mu_t,\mu_s)& \leq &\int_t^s \left |\frac{\!\dd \mu}{\!\dd \tau}\right |(\tau)\dd\tau, \: \forall s\geq t, \label{fco}
\end{eqnarray}}
and thus $\mu$ is a Cauchy curve for the Monge-Kantorovich distance. Since $\R^d$ is complete so is $\PP_2(\R^d)$ (see \cite[Proposition 7.1.5, p.154]{ags}) 
and $t \mapsto \mu(t)$ converges to some $\bar \mu$ as $t$ goes to infinity.$\hfill\Box$

\smallskip

If we did not have $\lim_{t\to+\infty} \J [\rho(t)]=0$, property \eqref{eq:lg} would imply that the subgradients along $\rho(\cdot)$ are bounded away from zero and thus there would exist a positive constant $c>0$ such that $|{\!\dd}\rho/{\!\dd t}|>c$. This would contradict the integrability property \eqref{macha}.

Thus $\lim \J[\rho(t)]=\hat \J=0$. Since $\J$ is lower semi-continuous the limit $\bar \rho$ of $\rho$ satisfies $\bar \rho\in \argmin\J$.

Combining \eqref{macha} and \eqref{fco}, we obtain
\begin{equation}\label{last}
  \M_2(\rho_0,\argmin \J)\leq \M_2(\rho_0,\bar \rho) \leq   \int_0^\infty \left |\frac{\!\dd\rho}{\!\dd t} \right|(t) \dd t  \le\frac{1}{\cg \theta} \J^{\theta}[\rho_0]
\end{equation}
which was the stated result with $\cf=(\theta\cg)^{{1}/{\theta}}$.

\paragraph{$\bullet$ Proof of (ii).} To establish \eqref{eq:lf} $\Rightarrow$ \eqref{eq:lg}, we further assume that, $\lambda=0$, {\it i.e.}, $\J$ is convex. If $\rho \in \argmin \J$ there is nothing to prove.  Assume that $\J[\rho]\in (0,r_0)$, take $\nu$ in $\partial \J[\rho]$. By Lemma~\ref{convineq}, for any $\bar \rho$  in  $\argmin \J$, we have
\begin{equation*}
    \J[\rho]\le  \M_2(\rho,\bar\rho) \left\|  \nu \right\|_{\rho},\;
  \end{equation*}
and thus
\begin{equation*}
    \J[\rho]\le  \M_2(\rho,\argmin \J) \left\|  \nu \right\|_{\rho}.
  \end{equation*}
By  \L ojasiewicz functional property~\eqref{eq:lf}, we obtain
\begin{equation*}
    \J[\rho]\le  \left(\frac{1}{\cf} \J[\rho] \right)^{\theta}\left\|  \nu \right\|_{\rho},
  \end{equation*}
or equivalently
   $\cf^{\theta} \J[\rho]^{1-\theta}\le \left\|  \nu \right\|_{\rho}.$
This is the claimed result  with $\cg=\cf^{\theta}$.

 This concludes the proof of Theorem~\ref{thm:main}. 
$\hfill\Box$


\bigskip

\noindent
Let us now proceed with the study of subgradient curves.\\

\noindent
{\bf Proof of Theorem~\ref{t:conv}} In view of the estimation of convergence rates, observe that inequality~\eqref{last} implies
\begin{equation}\label{majW}
  \M_2(\rho(t),\bar \rho) \leq   \int_t^\infty \left |\frac{\!\dd\rho}{\!\dd t} \right |(t) \dd t  \le\frac{1}{\cg \theta} \J^{\theta}[\rho(t)]
\end{equation}
and  $\bar \rho=\rho_{\infty}$.
From the above results 
$\lim_{t\to \infty} \J[\rho(t)]=\hat \J=0$ and $\rho$ converges to a \chg{minimiser $\rho_\infty$ of $\J$.  
Using \eqref{estim} and  applying  once more \eqref{eq:lg}, we obtain}
\begin{equation}\label{estim2}
  -\frac{\!\dd}{\!\dd t} \left\{\J[\rho(t)]^\theta \right\} \ge \, \cg^2 \theta \, \J[\rho(t)]^{1-\theta}.
\end{equation}
Setting $z(t)=\J[\rho(t)]^\theta$ for $t\in (0,\bar t)$, this gives the following differential inequality 
$$-\dot z(t)\geq \cg^2\theta z(t)^{1/\theta - 1 } \text{ on $[0,\bar t)$}.$$
Integrating the above inequality we obtain the desired estimates for $\J[\rho(t)]$. Those for the Monge-Kantorovich distance follow from \eqref{majW}.
\section{Applications to functional inequalities}\label{s:ineq}

\subsection{Relative internal energies}\label{sec:nrj}
Let   $V:\R^d\to \R\cup \{+\infty\}$ be a  lower semi-continuous convex potential such that 
\begin{align*}
&\mathcal{U}=\mbox{int} \,\dom V\neq \emptyset,\\
&\int_{\R^d}\exp(-V)=1.
\end{align*} This defines a log-concave probability measure $\rho^*:=\exp(-V)$. Consider a lower semi-continuous convex function $f:[0,+\infty)\to [0,+\infty)$  with $f(1)=0$ and such that  the map 
\begin{equation}\label{pd}
s\in(0,+\infty)\mapsto f(s^{-d})s^d\mbox{ is convex and non-increasing.}  
\end{equation}  Standard examples are $f(s)=s\log s$ or $f(s)=(s^m-s)/(m-1)$ with 
$m\geq 1-{1}/{d}$. The relative internal energy is defined as 

$$\displaystyle\J[\rho]:=  \left\{
\begin{aligned}
 & \int_{\R^d} f(\rho/\rhor)\dd\rhor \mbox{ if $\rho$ is absolutely continuous w.r.t $\dd\rhor$}\\
&+\infty \mbox{ otherwise.}
\end{aligned}
\right.
$$
\chg{Note that $\hat \J=0$ by Jensen's inequality.} 
It is known that $\J$ is lower semi-continuous and   convex  in  $\PP_2(\R^d)$ \cite[Theorem~9.4.12, p.224]{ags}. Let us introduce $P_f(s)=sf'(s)-f(s)$ for $s\geq0$. Denote by $L_\rho^2(\R^d)$ the space of square $\rho$-integrable functions in $\R^d$. By \cite[Theorem 10.4.9, p.265]{ags} $$\dom\partial \J=\left\{\rho\in \PP_2(\R^d):P_f(\rho)\in W^{1,1}_{\rm loc}(\mathcal{U}), \,\frac{\rho^*}{\rho}\nabla \left(P_f(\rho/\rhor)\right) \in L_\rho^2(\R^d)\right\}$$
and 
\begin{equation}
\|\partial^0 \J[\rho]\|_{\rho}^2=\int_{\R^d}  \left| \frac{\rho^*}{\rho}\nabla \left(P_f(\rho/\rhor)\right)\right|^2\!\dd\rho.
\end{equation}
In this context \L ojasiewicz gradient inequality \eqref{eq:lg} would write, for all $\rho \in \PP_2(\R^d)$:
\begin{equation*}
\J[\rho]<r_0 \Rightarrow \sqrt{\int_{\R^d}  \left| \frac{\rho^*}{\rho}\nabla \left(P_f(\rho/\rhor)\right)\right|^2\dd\rho} \geq \cg\left( \int_{\R^d} f(\rho/\rhor)\dd\rhor \right)^{1-\theta},   \end{equation*}
for some $\cg>0$, $r_0\in (0,+\infty]$ and $\theta\in (0,1]$. 
Theorem~\ref{thm:main} asserts that this inequality is equivalent to the functional \L ojasiewicz inequality \eqref{eq:lf}, for all $\rho \in \PP_2(\R^d)$:
\begin{equation}
\J[\rho]<r_0 \Rightarrow \cf\M_2(\rho,\argmin \J)^{{1}/{\theta}}\leq  \int_{\R^d} f(\rho/\rhor)\dd\rhor,
\end{equation}
for some $\cf>0$.
\smallskip

Whether such inequalities are satisfied for a general $f$ is not clear. However in the case of the Boltzmann  entropy  {\it i.e.,} $f(s)=s\log s$ much more can be said.


\subsubsection*{The logarithmic Sobolev inequality is equivalent to Talagrand inequality}
Consider the case when $f(s)=s \log s$, $V:\R^d\to \R$, a $C^2$-function with $\nabla^2V\geq K I_d$ and $K>0$. Hence for all $\rho$ in the domain of $\J$, 
$$\J[\rho]=\displaystyle \int_{\R^d}\log \left(\frac{\rho}{\rhor}\right) \dd\rho.$$
 We have 
$$\dom \partial \J= \left\{\rho\in W^{1,1}_{\rm loc}(\R^d),\, \nabla \log \left(\frac{\rho}{\rhor} \right)\in L_{\rho}^2(\R^d)\right\}.$$\chg{In this case the \L ojasiewicz gradient inequality takes the form of  the following logarithmic Sobolev inequa\-li\-ty, see e.g., \cite[Formula (9.27), p.279]{villani03}:
\begin{equation*}
   \int_{\R^d}  \left|\nabla \left[ \log \left(\frac{\rho}{\rhor}\right)\right]\right|^2\dd \rho \ge c_g^2 \left| \int_{\R^d}\log \left(\frac{\rho}{\rhor}\right) \dd \rho\right|, \forall \rho \in \dom \partial \J,
\end{equation*}
corresponding to $\theta={1}/{2}$, and $r_0=+\infty$. The optimal constant $c_g$ is given by $\sqrt{2K}$. 

On the other hand the functional \L ojasiewicz inequality  \eqref{eq:lf} is exactly a Talagrand type inequa\-li\-ty:
\begin{equation*}
 c_f \M_2(\rho, \rhor)^2 \le \int_{\R^d} \log \left(\frac{\rho}{\rhor}\right)\dd \rho, \forall \rho \in L^1(\R^d)\;\text{ with $\cf>0$}, 
\end{equation*}
where the optimal constant is $c_f=K/2$. 

Therefore Theorem~\ref{thm:main} ensures that, up to a multiplicative constant, the Talagrand and logarithmic Sobolev inequalities are equivalent 
\chge{under a convexity assumption}. This result is known and due to Otto-Villani, see \cite{ottovil}; it was obtained by completely different means\cite{\chge{Note also that our equivalence does not provide the non-convex equivalence they obtain for potentials whose Hessian are bounded from below}}.   
Remark~\ref{r:constant} shows further that  \eqref{eq:lg}, with the optimal constant $c_g=\sqrt{2K}$, implies \eqref{eq:lf} with the constant 
$$c_f=\left(\frac12 \sqrt{2K}\right)^2=\frac{K}{2},$$
which happens to be the optimal constant of the Talagrand inequality. For the reverse implication, the constant $c_g$ is half the optimal constant. }

\smallskip

The corresponding gradient system is the classical linear Fokker-Planck equation describing the evolution of the density within Ornstein-Uhlenbeck process:
\begin{equation*}
  \frac{\!\dd }{\!\dd t }\rho=\Delta \rho+\rm{div}\,(\rho \chg{\nabla V}), \; \rho(0)\in \dom \J.
\end{equation*}
Exponential  stabilisation rates are of course recovered through \chge{Theorem~\ref{t:conv}~(ii)}. 

\subsection{Sum of internal and potential energies}
\chg{
Let $F:[0,+\infty)\to \R$ be a convex differentiable function with super-linear growth satisfying $F(0)=0$, \eqref{pd} and 
$$\liminf_{s\to 0} s^{-\alpha}\,F(s)>-\infty, \mbox{ for some $\alpha>d/(d+2)$.}$$
Consider $V:\R^d\to \R$ a differentiable, $\lambda$-convex function with $\lambda\in \R$. 
Set
\begin{equation*}
\displaystyle \J[\rho]:=
\left\{\begin{aligned}
&\int_{\R^d}F(\rho)\dd x+\int_{\R^d}V\dd \rho \mbox{ if $\rho$ is absolutely continuous w.r.t the Lebesgue measure,}\\
&  +\infty \mbox{ otherwise.}
\end{aligned}\right.
\end{equation*}
The function $\J$ is lower semi-continuous,    
$\lambda$-convex  in  $\PP_2(\R^d)$ -- combine indeed \cite[Theorem~9.3.9, p.212]{ags} and \cite[Proposition~9.3.2, p.210]{ags}. Moreover  $$\dom\partial \J=\left\{\rho\in \PP_2(\R^d):P_F(\rho)\in W^{1,1}_{\rm loc}
(\R^d), \nabla F'(\rho)+ \nabla V \in L_\rho^2(\R^d)\right\},$$
see \cite[Theorem 10.4.13 p.273]{ags}.

In this context \L ojasiewicz gradient inequality \eqref{eq:lg} would write, for all $\rho$ in $ \PP_2(\R^d)$:
\begin{equation*}
 \hat\J <\J[\rho]< r_0\Rightarrow \sqrt{\int_{\R^d}  \left| \nabla F'(\rho)+ \nabla V \right|^2\dd\rho} \ge \cg\left( \int_{\R^d}F(\rho)\dd x+\int_{\R^d}V\dd \rho\right)^{1-\theta},  
\end{equation*}
for some $\cg>0, r_0\in (\hat \J,+\infty]$ and $\theta\in (0,1]$. Our main Theorem~\ref{thm:main} asserts that this inequality is equivalent to the functional \L ojasiewicz inequality \eqref{eq:lf}, for all $\rho\in \PP_2(\R^d)$:
\begin{equation*}
\hat\J <\J[\rho]< r_0\Rightarrow \cf \M_2(\rho,\argmin \J)^{{1}/{\theta}}\leq  \int_{\R^d} F(\rho)\dd x+\int_{\R^d}V\dd \rho,
\end{equation*}
for some $\cf>0$.
\smallskip

General conditions for the validity of theses inequalities are, up to our knowledge, not known at this day. Yet they hold for two specific choices of functions: $F(s)=s \log s$ and $F(s)=s^m/(m-1)$ for $m \ge 1-1/d$, corresponding respectively to the Boltzmann and non-linear entropies. \L ojasiewicz inequalities for the case involving  the Boltzmann entropy boil down to logarithmic Sobolev and Talagrand inequalities (see \chge{Section~\ref{sec:nrj}}).

The case of the ``power-entropy" is developed below. 
\subsubsection*{The Gagliardo-Nirenberg inequality is equivalent to a non-linear Talagrand type inequality} 
We assume  further $V$ is a $C^2$-function with $\nabla^2V\geq KI_d$ and  $K>0$.
 Let $d\ge 1$ and consider  $F(s)=s^m/(m-1)$, with $m \ge 1-1/d$. The unique minimiser of $\J$ is a Barenblatt profile, see~\cite{vazquez}
 \begin{equation*}
   \rhor(x)=\left(\sigma-\frac{m-1}{m}V(x)\right)^{{1}/{(m-1)}}_+ \forall x\in \R^d,
 \end{equation*}
 where $\sigma>0$ is such that $\displaystyle\int_{\R^d} \rhor=1$. Besides, we have
 $$\dom \partial \J=\left\{\rho\in \PP_2(\R^d):\rho \in W^{1,m}_{\rm loc}(\R^d), \nabla \left(\frac{m}{m-1}\rho^{m-1}+ V\right) \in L_\rho^2(\R^d)\right\}$$ and 
 $$\|\partial \J^0[\rho]\|_{\rho}^2= \int_{\R^d}\left |\frac{m}{m-1} \nabla (\rho^{m-1}) + \nabla V \right |^2\dd\rho.$$
When $r_0=+\infty$ and $\theta=1/2$ inequality \eqref{eq:lg}  writes
\begin{equation}\label{lgm}
\int_{\R^d}\left |\frac{m}{m-1} \nabla (\rho^{m-1}) + \nabla V \right |^2\dd\rho\ge \cg^2\left( \int_{\R^d} \frac{\rho^{m}}{m-1}\dd x+ \int_{\R^d}  V\dd \rho \right),
\end{equation}
 while ~\eqref{eq:lf} is
 \begin{equation}\label{ohta}
\int_{\R^d} \frac{\rho^{m}}{m-1}\dd x+ \int_{\R^d}  V\dd \rho  \ge \cf \M_2 (\rho,\rhor)^{2}.
 \end{equation}
   By \cite{dolbeault,under} inequality \eqref{lgm} is an instance of Gagliardo-Nirenberg inequality and holds true for \chge{an optimal} $\cg=\sqrt{2K}$. Therefore \eqref{ohta} holds true  for $c_f=K/2$: 
 \begin{equation}\label{CKOI}
  \int_{\R^d}\frac{\rho^{m}}{m-1}\dd x+\int_{\R^d}V\dd\rho\ge \frac{K}{2} \M_2 (\rho,\rhor)^{2}\,, \: \forall \rho\in W^{1,1}_{\rm loc}(\R^d).  \end{equation} 
   This inequality was recently obtained by Ohta-Takatsu in \cite{ohta} by completely different means. Our theorem also ensures that  Ohta-Takatsu inequality implies    Gagliardo-Nirenberg inequality  up to a multiplicative constant.

   \begin{remark}
Inequality  \eqref{CKOI} for the $\M_1$ distance is characterised in~\cite{bobkov} through the moments of the invariant measure. 
   \end{remark}
   
   Observe finally that the application of \chge{Theorem~\ref{t:conv} (ii)} in this framework allows to recover convergence rate results of nonlinear Fokker-Planck/porous medium dynamics   for $m\geq 1-1/d$: 
$$\frac{\!\dd }{\!\dd t }\rho=\Delta \rho^m+\rm{div}\,(\rho \chg{\nabla V}), \; \rho(0)\in \dom \J,$$
see e.g., \cite{dolbeault,carrillo00,carrillo06}.}

\subsubsection*{Logarithmic Sobolev and Talagrand's inequalities: the non-convex case }
\chg{We provide here an important example of a non-convex functional $\J$. Consider $F(s)=s\log s$,  and assume that $V:\R^d\to \R$  is of the  form 
$$V=V_1+V_2,$$
where $V_1$ is $C^2$ with  $\nabla^2V_1\geq K I_d$, $K>0$, and $V_2$ is in $L^{\infty}(\R^n)$. Set  $\mbox{\rm osc}\,(V_2):=\sup V_2-\inf V_2<+\infty$ and $\rhor:=\exp(-V)$. We assume that $\rhor$ is in $\P_2(\R^d)$. Then one has a \chge{logarithmic Sobolev inequality, see Holley-Stroock's article \cite{157}},
\begin{equation}\label{logsobg}
\int_{\R^d}  \left| \nabla \log \rho+ \nabla V \right|^2\dd\rho \ge \cg^2\left(\int_{\R^d}\rho\log \rho\dd x+\int_{\R^d}V\dd \rho\right),  
\end{equation}
where $c_g^2=2K\exp(-\mbox{\rm osc}\,(V_2))$.
Theorem \ref{thm:main} (i) shows that \eqref{logsobg} 
implies
\begin{equation*}
\cf \M_2(\rho,\rho^*)^2\leq  \int_{\R^d} \rho\log \rho \dd x+\int_{\R^d}V\dd \rho,
\end{equation*}
with $c_f=K\exp(-\mbox{\rm osc}\,(V_2))/2$.  This implication was already proved in \cite{ottovil} for non-convex $V$, see also \cite{gigli2013}. A famous family of energies for which \eqref{logsobg} holds, includes  the double-well potential: $x\mapsto ax^4-bx^2$ where $a>0$ (see \cite{157}).  

The application of Theorem~\ref{t:conv} to this setting allows to recover  classical convergence results for Fokker-Planck dynamics in the absence of convexity.}

 \subsection{Potential energies with \chge{$\lambda$-convex} subanalytic functions} 
\chg{Let $V:\RR^d\to \R$ be a  $\lambda$-convex function, $\lambda\in \R$. Assume that there exist $a,b$ in $\R$ with $V(x)\geq -a|x|^2+b$. The functional, called {\it potential energy} is given by
$$\J[\rho]:=\int_{\RR^d}  V \dd \rho.$$
$\J$ is well defined,   lower semi-continuous, and $\lambda$-convex, see \cite[Proposition 9.3.2 p.210]{ags}. 

\smallskip

\subsubsection*{\chge{Lifted convex functions}} We first consider the case when $V$ is convex, {\it i.e.}, $\lambda=0$. We obviously have $\argmin \J=\{\rho \in \P_2(\R^d): \mbox{supp}\,\rho\subset \argmin V\}$ and $\hat \J=\hat V$. If $\mathcal{S}$ is a subset of $\R^d$, we set $\dist(x,\mathcal{S})=\inf \{|x-y|:y\in \mathcal{S}\}$.

In that case  \L ojasiewicz inequalities can be lifted  in the Monge-Kantorovich space. As a consequence, we will derive the following result:
\begin{theorem}[\L ojasiewicz inequality for a convex potential energy]\label{th3}
Let $V:\RR^d \to \RR$ be a convex function such that $\hat V:=\inf V>-\infty$. If $V$ is subanalytic  with compact level sets then there exist $c_f>0$, $c_g>0$, and $\theta\in (0,1]$ such that $\J$ satisfies the \L ojasiewicz inequalities \eqref{eq:lg} and  \eqref{eq:lf} with $r_0=+\infty$. 
\end{theorem}

The proof of the above theorem ensues from the fact that $V$ satisfies itself the \L ojasiewicz inequality with $r_0=+\infty$ (see e.g. \cite{bolte15}) and from the following equivalence result:

\begin{theorem}[\L ojasiewicz inequality in $\P_2(\R^d)$ is equivalent to \L ojasiewicz inequality in $\RR^d$]
We make the assumptions of Theorem~\ref{th3}. Fix $r_0\in (\hat V,+\infty]$. The two following assertions are equivalent:\\
(i) \chge{There exist $c>0$, $\theta \in (0,1]$ such that for each $x$ in $\R^d$,}
  \begin{equation*}
 \chge{ \hat V< V(x)< r_0} \quad \Rightarrow \quad  c \, \dist(x,\argmin V)^{{1}/{\theta}} \leq \,V(x)- \hat V.
  \end{equation*}
 (ii) \chge{There exist $c'>0$, $\theta' \in (0,1]$ such that for each $\rho$ in $\Pd$,} 
  \begin{equation}\label{eq:lf2}
  \supp \rho \, \subset V^{-1}(\hat V,r_0) \quad \Rightarrow \quad  \chge{c' \, \M_2(\rho,\argmin \J)^{{1}/{\theta'}} \leq \,\J[\rho]- \hat \J.}
  \end{equation}
\end{theorem}

\begin{proof}
$(ii)\Rightarrow (i)$. Take $x\in \R^d$ and write the inequality \eqref{eq:lf2} for $\delta_x$, the Dirac at $x$. The result follows from $\hat \J=\hat V$ and  
$$\dist (x, \argmin V)=\M_2 (\delta_x,\argmin \J).$$ 
\chge{Observe that this implication holds regardless of the value of $\theta$.}
%
%
%

\noindent
$(i)\Rightarrow (ii)$. Assume for simplicity that $\hat V=0$ and thus $\hat \J=0$. Using the \L ojasiewicz inequality for $V$, we have 
$$V(x)\geq c\,\dist (x,\S)^{1/\theta}  \mbox{ for all }x \mbox{ such that }V(x)<r_0,$$
where $\S=\argmin V$. \chge{As a consequence}
\begin{equation}\label{eq:prem}
c^{-1}\J[\rho] \geq \int_{\R^d} \dist(x,\S)^{1/\theta}\dd\rho =  \int_{\R^d} |x-\proj_\S(x)|^{{1/\theta}}\dd\rho\notag
 \end{equation}
\chge{If $\theta \le 1/2$, using H\"older's inequality, we obtain 
\begin{equation}\label{eq:le12}
c^{-1}\J[\rho]  \geq   \left(\int_{\R^d} |x-\proj_\S(x)|^2\dd\rho\right)^{{1/(2\theta)}}. 
 \end{equation}
For $\theta>1/2$,  we need the following Gagliardo-Nirenberg inequality: 
\begin{equation*}
  \|u\|_p \leq C\|u\|_q^{1-a}\|u\|_{W^{1,r}}^a\quad\mbox{with}\quad  a:=\frac{1/q-1/p}{1/q+1/d-1/r}, 
\end{equation*}
$1\le q \le p \le \infty$,  $r>d$ and where $C=C(p,q,d)$. Applying the above  to $u={\rm id}-\proj_\S$ with $p=2$ and $q=1/\theta$ entails 
\begin{equation}\label{eq:qe12*}
c^{-1}\J[\rho]  \geq  \gamma \left(\int_{\R^d} |x-\proj_\S(x)|^2\dd\rho\right)^{{1/(2\theta(1-a))}}
 \end{equation}
 where $\gamma>0$ is an adequate constant.
For any $\theta \in (0,1]$, we can assemble~\eqref{eq:le12} and~\eqref{eq:qe12*} into
\begin{equation}\label{eq:qe12**}
(c')^{-1}\J[\rho]  \geq \left(\int_{\R^d} |x-\proj_\S(x)|^2\dd\rho\right)^{\alpha} 
 \end{equation}
where $\alpha:=1/2\theta$ and $c=c'$ if $\theta \le 1/2$, while $\alpha:=1/(2\theta(1-a))$  and $c'=\gamma c$ otherwise. We finally obtain
\begin{eqnarray*}
(c')^{-1}\J[\rho]  & \geq &  \inf \left\{ \int_{\R^d} |x-T(x)|^2\dd\rho: \,T \mbox{ {\small Borel map on $\R^d$ with} } T(\supp \rho )\subset \S\right\}^{\alpha}\notag\\
 & = &   \M_2(\rho,\argmin \J)^{2\alpha}
 \end{eqnarray*}
When $\theta\in(0,1/2]$, one obtains the result with $\theta'=\theta$ and $c'=c$. Otherwise, for $\theta>1/2$, we choose
\begin{equation}\label{exponent}\theta'=\theta\frac{1/2+1/d}{\theta+1/d}\in (0,1),\end{equation}
which concludes the proof.}\end{proof}

\medskip
\chge{\begin{remark} (a)  With the same assumptions and for $\theta\in (0,1/2]$, the equivalence  between $(i)$ and $(ii)$ holds with $\theta=\theta'$ and $c=c'$.\\
(b)  In the case when $\theta \in (1/2,1]$, the curvature of the Monge-Kantorovich space deteriorates the lifted exponent $\theta'$ by flattening 
the profile $\J$. To see this, consider for $d=1$, the sharp function $V(x)=|x|$ whose 
exponent is $\theta=1$ and the path $(\rho_t)_{t\in[0,1]}=(1-t)\delta_0 + t\delta_1$ in $\P_2(\R^d)$, then $\J[\rho_t]=t$ while $\M_2(\rho_t,\argmin 
\J)=\M_2(\rho,\delta_0)=\sqrt t\,$ for all $t$ in $[0,1]$. Hence the optimal exponent of $\J$ is lower than $1/2$ and thus $\J$ does not inherit of the sharpness of its kernel $V$.

Observe that, in the general case,  the estimation \eqref{exponent}  predicts indeed a lifted exponent with a value lower than $\theta$:
$$\theta'=\theta\,\frac{1/2+1/d}{\theta+1/d}<\theta.$$
\end{remark}}

\bigskip

\subsubsection*{\chge{Non-convex kernels}} When $V$ is non-convex but merely $\lambda$-convex for $\lambda$ in $\RR$, we observe a mixing of critical values of the energy. It is easily seen by considering two arbitrary critical points $x,y$ of $V$ and by considering the segment of critical measures $s\mapsto s\delta_x+(1-s)\delta_y$ whose image by $\J$ is the segment $[V(x),V(y)]$. This shows that the set of critical values of $\J$ is given by the convex envelope of the critical values of $V$, implying a quasi-systematic failure of Sard's theorem. Hence formulating \L ojasiewicz inequalities on slice of level sets is no longer relevant.

 Instead we proceed as follows: 

\begin{lemma}[Lifting the gradient inequality on $\P_2(\R^d)$]\label{loyacV}
 Let $\bar V\in \R$ and $X\subset\R^n$.  Consider $\theta \in (0,1]$ and $c>0$.  If 
\begin{equation}\label{loya1}
   \left| V(x) -\bar V\right|^{1-\theta} \le c\, \left| \nabla V(x)\right|, \mbox{ for all $x$ in } X, 
  \end{equation}
then 
  \begin{equation*}
   \left | \J[\rho] -\bar V\right |^{1-\theta} \le c \left\| \partial \J[\rho]\right\|_\rho,
  \end{equation*}
  for all $\rho$ in $\PP_2(\R^d)$ such that $\supp \rho\subset X$.
  \end{lemma}
\begin{proof}

Take $\rho$ such that $\supp (\rho) \subset X$. By H\"older's inequality and~\eqref{loya1}, we have
\begin{eqnarray*}
    \left | \J[\rho]-\bar V\right |  & =  & \left| \int_X \rho(x)(V(x)-\bar V)\dd x\right|\\
    &  \le & c^{1/(1-\theta)}\int_X \rho(x)\left| \nabla V(x)\right|^{1/(1-\theta)}\dd x\\
        &  \le & c^{1/(1-\theta)}\left[\int_X \rho(x)\left| \nabla V(x)\right|^2\dd x\right]^{1/(2(1-\theta))}\\
    & = & c^{1/(1-\theta)}\left\| \partial \J[\rho]\right\|_\rho^{1/(1-\theta)}.
\end{eqnarray*}
Which is the stated result.
\end{proof}

As a consequence we obtain the following general \L ojasiewicz gradient inequality for the class of potential energies.

\begin{proposition}[\L ojasiewicz gradient inequalities for potential energies]
Let $V:\R^d\to \R$ be a differentiable subanalytic  function and $\C$ a connected component of~$\nabla V^{-1}(0)$. 

(i) $\J$ is constant on the set of measures having support in $\C$,  we denote by $\bar \J$ this value.

(ii) Fix  $R>0$. There exist $\epsilon>0$, $c>0$, and $\theta\in (0,1]$ such that
 \begin{equation*}
   \left | \J[\rho] -\bar \J\right |^{1-\theta} \le c \left\| \partial \J[\rho]\right\|_\rho,
  \end{equation*}
  for all $\rho$ in $\PP_2(\R^d)$ such that $\supp \rho \subset \left\{x\in \R^d: \dist(x,\C)< \epsilon, |x|\leq R \right\}$.
  
\end{proposition}
\begin{proof} Standard results on subanalytic geometry can be found in \cite{loja65}. Observe first that, by subanaliticity and continuity, $V$ must be constant on $\C$. We then combine the classical \L ojasiewicz inequality  and a compactness argument to obtain a uniform inequality on $$X:=\left\{x\in \R^d: \dist(x,\C)< \epsilon, |x|\leq R \right\}$$ for some $\epsilon>0$. The conclusion follows by Lemma~\ref{loyacV}.\end{proof}

Potential energies provide straightforward examples of convergence of subgradient flows of the form (see \cite[Example 11.2.2, p. 298]{ags}):
\begin{eqnarray}
\frac{\!\dd}{\!\dd t} \rho(t)+\nabla\cdot(\rho(t)v(t)) = 0& & \mbox{a.e. on }(0,+\infty),\\
v(t)\in \partial V (\rho(t))& &\mbox{a.e. on }(0,+\infty),
\end{eqnarray}
covering all possible convergence rates  since $\theta$ may range in $(0,1]$ (use Theorem \ref{t:conv}).  For instance when the norm of the subgradients of $V$ are bounded away from zero at each point, save of course at minimisers, the functional $\J$ is sharp and convergence of the associated flow occurs in finite time.

\section{Appendix: Notations and fundamental results}\label{s:app}
Let us remind here a few elements of formal geometry of the probability measures with the Monge-Kantorovich distance. General monographs on the subject are \cite{villani03,ags,users,santambrogio15}.
\subsection{Monge and Kantorovich's problems}
Let $X$ and $Y$ be two metric spaces equipped res\-pec\-tively with the Borel probability measures $\mu\in \P(X)$ and $\nu\in \P(Y)$. For $\mu\in \P(X)$ and a Borel map $T:X\to Y$, $T_\#\mu$ denotes the {\it push-forward} of $\mu$ on $\nu$ through $T$ which is defined by $T_\#\mu(B)=\mu(T^{-1}(B))$ for every Borel subset $B$ of $Y$ or equivalently by the change of variables formula
\begin{equation*}
\int_Y \varphi(x) \dd T_\#\mu(x)=\int_X \varphi(T(x)) \dd\mu(x), \; \forall \varphi \in \C_b(X). 
\end{equation*}
A transport map \chge{$T:X\to Y$} between $\mu$ and $\nu$ is a Borel map such that $T_\#\mu=\nu$. 

Assume $X=Y$ and denote by $d$ the distance on $X$. For $(\mu,\nu)\in \P(X)\times \P(X)$  the Monge optimal transport problem  writes  
\begin{equation}
\M_2(\mu,\nu)=\sqrt{\inf \left\{\int_X \chg{d^2(x,T(x))} \dd\mu(x):\mbox{$T$ Borel map s.t. }T_\#\mu=\nu\right\}}.\label{monge}
\end{equation}
$\M_2$ defines a distance on the subset of probabilities on $X$ with finite second-order moments.
\chg{Given $\mathcal{S}\subset \P(X)$ and $\mu$ in $\P(X)$, we set $\M_2(\mu,\mathcal{S})=\inf\left\{\M_2(\mu,\nu):\nu \in \mathcal{S}\right\}$.}

\bigskip

\noindent
From now on, we assume $X=Y=\R^d$ where $d$ is a positive integer. Set 
\begin{equation}
\P_2(\R^d)=\left\{\mu\in \P(\R^d): \int_{\R^d} |x|^2\dd\mu(x)<+\infty \right\}.
\end{equation}
A solution to \eqref{monge} is called {\em an 
 optimal transport}. Such a transport generally exists thanks to:
\begin{theorem}[Brenier's theorem \cite{brenier1991polar}]\label{brenierthm}
 Consider $(\mu,\nu) \in \P_2(\R^d)^2$  and assume that $\mu$ is regular in the sense that each hyper-surface has a null measure. Then the Monge optimal transport pro\-blem has a unique solution $T$ \chg{to~\eqref{monge}}. 
Moreover $T=\nabla u$ $\mu$-a.e. for some convex function $u:\R^d\to\R$ and $\nabla u$ is the unique (up to $\mu$-a.e. equivalence)  gradient of a convex function transporting $\mu$ onto~$\nu$.
\end{theorem}

In the absence of regularity of the source measure, Monge's formulation can be relaxed into the so-called Kantorovich's formulation 
\[\M_2^2(\mu, \nu):=\min_{\gamma\in \Pi(\mu, \nu)} \iint_{\R^d\times \R^d} \vert x-y\vert^2 \dd\gamma(x,y)\] 
where $\Pi(\mu, \nu)$ is the set of measures in $\R^d\times \R^d$ whose first and second marginals are respectively $\mu$ and $\nu$. This problem always has a solution called an optimal transport plan; the set of such plans is denoted by $\Gamma_{\text{\rm opt}}(\mu,\nu)$.

\subsection{Convexity and geodesics}

Let us consider  real-extended-valued functions $\J:\P_2(\R^d)\to (-\infty,+\infty]$ for which we set $\dom \J=\{\rho \in \P_2(\R^d):\J[\rho]<+\infty\}$. These functionals are called {\em proper} when $\dom \J\neq\emptyset$.

\smallskip

\chg{
For $\rho$, $\mu$, $\nu$ in $\P_2(\R^d)$, we denote by $\Pi(\rho,\mu, \nu)$ the set of measures in $\R^d\times \R^d\times \R^d$ whose first, second and third marginals are respectively $\rho$, $\mu$, and $\nu$. 

\smallskip

From \cite[Definition 9.2.4, p.~207]{ags} we recall:
\begin{definition}[Generalised geodesics and $\lambda$-convexity] \label{defcgg} Set $\pi_{1}(x,y,z)=(x,y)$ and $\pi_{2}(x,y,z)=(x,z)$ for all $x,y,z$ in $\R^d$. 
\begin{itemize} 
\item[(i)]
Let $\rho,\mu_1,\mu_2$ be in $\P_2(\R^d)$. Take  $\Psi\in \Pi(\rho,\mu_1,\mu_2)$ such that $(\pi_{1})_{\#}\Psi\in \go(\rho,\mu_1)$,  $(\pi_{2})_{\#}\Psi\in \go(\rho,\mu_2)$. 
Define 
$$\pi^t=(1-t)\pi_1+t\pi_2, \: t\in [0,1],$$
then the generalised geodesic joining $\mu_1$ to $\mu_2$ with base $\rho$ induced by $\Psi$ is given by:
$$[0,1]\ni t \to \pi^t_{\#}(\Psi).$$
\item[(ii)] Let $\lambda\in \R$. A functional $\J$:  $\P_2(\R^d) \to (-\infty,+\infty]$ is called $\lambda$-\emph{convex along 
generalised geodesics} if for every pair $(\nu_0,\nu_1)\in\dom \J\times \dom \J$ and any generalised geodesic  $\{\nu_t\}_{t\in[0,1]}$  with base $\rho\in \P_2(\R^d) $ induced by $\Psi\in \Pi(\rho,\nu_0,\nu_1)$, one has

\[\J[\nu_t]\leq (1-t) \J[\nu_0]+t\J[\nu_1]-\frac{\lambda}{2}\mathcal{W}^2_{2,\Psi}(\nu_0,\nu_1)\]
for every $t\in[0,1]$ and where
$$\mathcal{W}^2_{2,\Psi}(\nu_0,\nu_1):=\iiint_{\R^d\times \R^d\times \R^d}|x_3-x_2|^2\dd\Psi(x_1,x_2,x_3).$$  
%
\end{itemize}
\end{definition}

\begin{remark}The above definition implies in particular that the domain of $\J$ is convex in the following sense: for any regular measure $\gamma$, and for any measures $\mu$ and $\nu$ in  $\dom \J$, the generalised interpolant
with base $\gamma$ between $\mu$ and $\nu$ exists and lies in $\dom \J$.
\end{remark}
\smallskip

When $\lambda=0$, which is a major focus within this article, the definition simplifies into:
%
%

\begin{definition}[Generalised  convexity] 
A functional $\J$:  $\P_2(\R^d) \to (-\infty,+\infty]$ is called  \emph{convex} (along generalised geodesics)\footnote{Convexity in the Monge-Kantorovich spaces is sometimes referred to as {\em displacement convexity}.} if for every pair $(\nu_0,\nu_1)\in\dom \J\times \dom \J$ and any generalised geodesic  $\{\nu_t\}_{t\in[0,1]}$ between $\nu_0$ and $\nu_1$, one has
\[\J[\nu_t]\leq (1-t) \J[\nu_0]+t\J[\nu_1]\]
for every $t\in[0,1]$.
\end{definition}
}
\subsection{\chg{Metric/Riemannian aspects, \cite{otto01,ags}}}
\chg{We first recall a purely metric notion: a curve $\gamma:(a,b)\to \P_2(\R^d)$ is called absolutely continuous if 
$$\M_2(\gamma(t),\gamma(s))=\int_t^s v(\tau)\dd\tau,\: \forall t,s \in (a,b) $$
where $v$ is $L^1(\R)$. The metric derivative of $\gamma$ at $t\in (a,b)$ is given by 
\begin{equation*}
\left|\frac{\!\dd \gamma }{\!\dd t}\right|(t)=\lim_{h\to 0} \frac{\M_2(\gamma(t+h),\gamma(t))}{|h|}.
\end{equation*}
This quantity is well defined for almost every $t\in (a,b)$, see~\cite[Theorem 1.1.2, p.24]{ags}.}

\smallskip

Let $\rho$ be in $\P_2(\R^d)$, the tangent space to $\P_2(\R^d)$ at $\rho$, written $T_\rho\P_2(\R^d)$, is identified to the subspace of distributions formed by the vectors $s=-\nabla\cdot(\rho\nabla u)$ where $u$ ranges over 
$C^{\infty}(\R^d,\R)$.
The scalar product of two vectors $s_1=-\nabla\cdot(\rho\nabla u_1)$, $s_2=-\nabla\cdot(\rho\nabla u_2)$, is given by
$$\langle s_1,s_2\rangle_\rho=\int_{\R^d}  \nabla u_1\cdot\nabla u_2 \dd\rho.$$
The associated norm is as usual \chg{$\|s\|_\rho:=\sqrt{\langle s,s\rangle_\rho}$. }

\medskip

Let  $\J:\P_2(\R^d)\to (-\infty,+\infty]$ be a   convex function. 
Define the metric (or strong) slope of $\J$ at $\rho\in \dom \J$  by 
$$|\nabla |\J[\rho]=\limsup_{\mu \to \rho}\frac{(\J[\rho]-\J[\mu])^+}{\M_2(\rho,\mu)} \in (-\infty,+\infty].$$
\chg{For the subdifferential of $\J$, we pertain to  the reduced subdifferential \cite[Definition 10.3.1, p.241 and Remark 10.3.3 p.243]{ags} 
which  also admits the equivalent formulation in the case of geodesically convex function, and thus in particular for convex functions in the sense of the previous definition:

\begin{definition}(\cite[Item 10.3.13 p. 243 and Theorem 10.3.6 p. 244]{ags} Reduced subdifferential) Let $\J:\P_2(\R^d)\to (-\infty,+\infty]$ be a $\lambda$-convex lower semi-continuous  function bounded from below with $\lambda\in \R$. Take 
$\mu\in \dom \J$ and  $\gamma \in L_{\mu}^2(\R^d)$. Then 
\begin{align*}
 &\gamma\in \partial \J[\mu]&&\Longleftrightarrow &&\forall \nu \in \P_2(\R^d), \exists\Psi \in \go(\mu,\nu), \\
 & &&&&\J[\nu]\geq \J[\mu]+\int_{\R^d\times \R^d} \gamma(x)\cdot (y-x)\dd \Psi(x,y)+\frac{\lambda}{2}\M_2^2(\mu,\nu).
\end{align*}
 When $\mu\notin \dom \J$, the set $\partial \J[\mu]$ is empty.
\end{definition}

One defines {\em the minimal norm subgradient}, whenever it exists, by
\begin{equation*}
\partial^0\J[\mu]=\argmin \{\|\gamma\|_{\mu}: \gamma\in \partial \J[\mu]\}.
\end{equation*}
Using \cite[Theorem 10.3.10, p.246]{ags}, we have $\dom \partial \J\subset \dom |\nabla|\J$ and
\begin{equation}\label{pente}
|\nabla|\J[\mu]\leq\|\partial^0 \J[\mu]\|,\quad \forall \mu \in \P_2(\R^d).
\end{equation}
}

\bigskip

\noindent
{\bf Acknowledgements} The second author thanks the Air Force Office of Scientific Research, Air Force Material Command, USAF, under grant number  FA9550-14-1-0500, the  FMJH Program Gaspard Monge in optimization and ANR OMS (ANR-17-CE23-0013-01)  for supporting his research.

The authors are indebted with T. Champion, N. Gozlan, F. Silva, S. Sorin and the referees for their very useful comments.
\bibliographystyle{siam}

\end{document}